\documentclass[10pt,oneside,reqno]{amsart}
\usepackage{amsmath,amssymb}
\usepackage{amsthm} 

\newtheorem{thm}{Theorem}[section]

\newtheorem{lemma}[thm]{Lemma}

\theoremstyle{definition}

\theoremstyle{remark}

\numberwithin{equation}{section}
\newtheorem{exa}[thm]{Example}

\newtheorem{todo}{ToDo}

\def\dif{\operatorname{d}}
\def\du{\dif u}
\def\a{\alpha}
\def\b{\beta}
\def\c{\gamma}
\allowdisplaybreaks

\begin{document}

\title[Asymptotic expansions of integral means]
{Asymptotic expansions of integral means and applications to the ratio of gamma functions}
\author{Neven Elezovi\'c and Lenka Vuk\v si\'c}

\address{Neven Elezovi\'c, Faculty of Electrical Engineering and Computing, University of
Zagreb, Unska 3, 10000 Zagreb, Croatia}
\email{neven.elez\@fer.hr}

\address{Lenka Vuk\v{s}i\'c, Faculty of Electrical Engineering and Computing, University
of Zagreb, Unska 3, 10000 Zagreb, Croatia}
\email{lenka.vuksic\@fer.hr}


\begin{abstract}
	Integral means are important class of bivariate means. In this
	paper we prove the very general algorithm for calculation of coefficients
	in asymptotic expansion of integral mean. It is based on explicit solving 
	the equation of the form $B(A(x))=C(x)$, where $B$ and $C$ 
	have known asymptotic expansions. The results are ilustrated by calculation of 
	some important integral means connected with gamma and digamma functions.
\end{abstract}



\maketitle

\section{Introduction}

	Let $f$ be a strictly monotone continuous function. Then there exists the unique $\vartheta\in[s,t]$ for which
	\begin{equation*}
		\frac1{t-s}\int_s^t f(u)\du=f(\vartheta).
	\end{equation*}
	$\vartheta$ is called {\it integral $f$-mean} of $s$ and $t$, and denoted by
	\begin{equation}
		I_f(s,t)=f^{-1}\left(\frac{1}{t-s}\int_s^t f(u) \du \right),
	\end{equation}
	see \cite{Bu,BMV} for details.

	Many  classical well known means can be interpreted as integral means, for suitably choosen function $f$.
	For example,
        \begin{equation*}
        \begin{aligned}
                f(x)&=x,&&      I_x(s,t)=\frac{s+t}2=A(s,t),\\
                f(x)&=\log x,&&I_{\log}(s,t)=\frac1e\left(
                \frac{t^t}{s^s}\right)^{\frac1{t-s}}=I(s,t),\\
                f(x)&=\frac1x,&&I_{1/x}(s,t)=\frac{s-t}{\log s-\log t}=L(s,t),\\
                f(x)&=\frac1{x^2},&&I_{1/x^2}(s,t)=\sqrt{st}=G(s,t),\\
                f(x)&=x^{r},\quad r\ne0,-1\qquad&&I_{x^{r}}(s,t)
		=\left(\frac{t^{r+1}-s^{r+1}}{(r+1)(t-s)}\right)
                ^{\frac1{r}}=L_r(s,t),\\
        \end{aligned}
        \end{equation*}
        where $A$, $I$, $L$, $G$, and $L_r$ are arithmetic, identric, logarithmic,
        geometric and generalized
        logarithmic mean.

	Although integral means are defined as functions of $x$, $s$ and $t$,
	it turns out that the easier notation will be obtained if we introduce the following two variables:
	\begin{equation*}
		\a=\frac{t+s}2,\qquad \b=\frac{t-s}2.
	\end{equation*}
	Then $t=\a+\b$ and $s=\a-\b$. 
	All asymptotic expansions will be given in terms of $\a$ and $\b$.

	In recent papers \cite{EV,EV2} asymptotic expansions of many bivariate means are found, using the explicit formulas for observed means. The coefficients of asymptotic expansions are very usefull in analysis of considered means. Here is a short table.
	General form of the expansion for a mean $M(s,t)$ is:
	\begin{equation*}
		M(x+s,x+t)=x+\a+\sum_{n=2}^\infty c_n x^{-n+1} 
	\end{equation*}
	and the first few coefficients are
	\begin{equation*}
	\def\st{\hbox{\vrule height 15pt depth6pt width0pt}}
	\begin{array}{|c|c|c|c|}
	\hline
	\text{mean}&c_2&c_3&c_4\st\\
	\hline
	\hline
	A&0&0&0\st\\
	\hline
	I&-\frac16\b^2&\frac16\a\b^2&-\frac1{360}\b^2(60\a^2+13\b^2)\st\\
	\hline
	L&-\frac13\b^2&\frac13\a\b^2&-\frac1{45}\b^2(15\a^2+4\b^2)\st\\
	\hline
	G&-\frac12\b^2&\frac12\a\b^2&-\frac18\b^2(4\a^2+\b^2)\st\\
	\hline
	L_r&\frac16(r{-}1)\b^2&-\frac16(r{-}1)\a\b^2&\frac1{360}(r{-}1)\b^2((-2r^2{-}5r{+}13)\b^2{+}60\a^2)\st\\
	\hline	
	\end{array}
	\end{equation*}

        In the paper \cite{EP} the connection between differential and integral $f$-mean of a function $f$ was analyzed and results were applied to digamma function. Among others, it was proved that
        \begin{equation*}
                \psi\biggl( \frac{t-s}{\log t-\log s}\biggr)
                \le\frac1{t-s}\int_s^t\psi(u)\dif u,
        \end{equation*}
        i.e.,
        \begin{equation*}
                L(s,t)\le I_\psi(s,t).
        \end{equation*}
        Here, $\psi$ denotes digamma function.
        As a consequence it follows that the function $x\mapsto I_\psi(x+s,x+t)-x$ is increasing concave function and
        \begin{equation*}
                I_\psi(x+s,x+t)-x\uparrow A(s,t),
                \qquad\text{as }x\to\infty.
        \end{equation*}

        This equation was essential in the proving the lower bound in the second Gautschi-Kershow inequalities, see \cite[Theorem 4]{EGP} for details:
        \begin{equation*}
                \exp[\psi(x+I_\psi(s,t))]
                <\biggl(\frac{\Gamma(x+t)}{\Gamma(x+s)}
                \biggr)^{\tfrac1{t-s}}<\exp[\psi(x+A(s,t))].
        \end{equation*}

        In a recent paper \cite{CEV} a complete asymptotic expansion of the function $G$ in the formula
        \begin{equation}
        \biggl(\frac{\Gamma(x+t)}{\Gamma(x+s)}\biggr)^{\tfrac1{t-s}}=    \exp[\psi(G(x))]
        \label{digamma}
        \end{equation}
        was found. The first few terms are
        \begin{equation}
        G(x)=c_0+\frac{c_1}x+\frac{c_2}{x^2}
        +\frac{c_3}{x^3}
        +\frac{c_4}{x^4}+\dots
        \label{G}
        \end{equation}
        where
	\begin{equation}
	\begin{aligned}
        c_0&=1,\\
        c_1&=\a,\\
        c_2&=-\frac{\b^2}6,\\
        c_3&=\frac1{12}\b^2(2\a-1),\\
        c_4&=-\frac1{360}\b^2\Big[60\a^2-60\a+13\b^2+5\Big].
        \end{aligned}
	\label{coef-psi}
	\end{equation}

        It is easy to see that the form \eqref{digamma} is equivalent to
        \begin{equation*}
                G(x)=I_\psi(x+s,x+t).     
        \end{equation*}
        Therefore, expansion \eqref{G} is the expansion of the integral mean of digamma function. Second Gautschi-Kershow inequality can be read as
        \begin{equation*}
                x+I_\psi(s,t)<I_\psi(x+s,x+t)<x+A(s,t).
        \end{equation*}
        The right side of this inequality is in fact the beginning of an asymptotic expansion. Using the result above, one can state the following inequality which still needs to be proved:
        \begin{equation}
                x+A(s,t)-\frac1{24}(t-s)^2\frac1x
                <I_\psi(x+s,x+t)<x+A(s,t).
        \end{equation}

	The goal of this paper is to generalize this approach from digamma function and other particular functions to an arbirary taken function $f$ (which should have some additional properties).

	The main point is that in the case of digamma function, an explicit formula for the inverse function is not known. Therefore, we cannot derive an explicit formula for integral mean of this function. So we need to derive an algorithm for calculation of these integral means using other methods.
        
        We shall assume that the function $f$ posses an asymptotic expansion of the form
	$$
		f(x)\sim x^u \sum_{n=0}^{\infty} b_n x^{-n},
	$$
	and we shall derive the asymptotic expansion of its integral mean
	$$
		I_f(x+s,x+t)\sim x \sum_{n=0}^{\infty} a_n x^{-n},
	$$
	where $u$ is a real number and $(b_n)$ is the given sequence.

\section{Solving an equation in terms of asymptotic series}

        In this section the general result concerning composition of asymptotic series will be derived.

	We are dealing with asymptotic series, because of the 
	intended applications. But the results of this section can be applied to formal power series as well. The problem is very old and it is treated for at least two hundred years.  However, to the best of our knowledge, we did not succeed to find in a literature the solution and algorithm in the form which is given here.

	The following lemma has its origin in Euler's work, see \cite{G} for historical treatment in the case of Taylor series. We used it already in the context of asimptotic series in \cite[Lemma 6]{CEV}:

\begin{lemma}
	\label{potl}
	Let $a_0\neq 0$ and $g(x)$ a function with asymptotic expansion 
	\begin{equation*}
                g(x)\sim \sum_{n=0}^{\infty}a_nx^{-n}.
	\end{equation*}
	Then for all real $r$ it holds
	\begin{equation*}
                [g(x)]^{r}\sim \sum_{n=0}^{\infty}P_n(r)x^{-n},
	\end{equation*}
	where	
        \begin{equation}
        \begin{aligned}
                P_0(r)&=a_0^{r},\\
                P_n(r)&=\frac{1}{na_0}\sum_{k=1}^{n}[k(1+r)-n]a_kP_{n-k}(r).
	\end{aligned}
        \label{pot}
        \end{equation}
\end{lemma}

	We shall use another one functional transformation of asymptotic series. The proof of the folowing lemma is easy, see \cite{CEV} for details.

\begin{lemma}
	\label{lemma-log}
    Let $a_0\ne 0$ and
    $$
        A(x)\sim\sum_{n=0}^\infty a_nx^{-n}
    $$
    be a given asymptotic expansion. Then the composition $L(x)=\ln(A(x))$ has
    asymptotic expansion of the following form
    $$
        L(x)\sim\sum_{n=1}^\infty L_n x^{-n}
    $$
    where
 \begin{equation}
        L_n=\frac{a_n}{a_0}-\frac1{na_0}\sum_{k=1}^{n-1} k L_k a_{n-k},\qquad n\ge1.
    \label{log} 
    \end{equation}
\end{lemma}

	If $A(x)$ and $B(x)$ are asymptotic series as $x\to\infty$, 
	one may ask when the composition $B(A(x))$ will also be an asymptotic series? This will not always be the case. For example, if $A(x)=1/x$, then $B(1/x)$ need not to be an asymptotic series. See \cite{Er} for additional discussion.

	Suppose $A(x)$, $B(x)$ and $C(x)$ have the following form
	\begin{align}
        \label{A}
		A(x)&\sim x^w\sum_{n=0}^\infty a_nx^{-n},\\
	\label{B}
                B(x)&\sim x^u\sum_{n=0}^\infty b_nx^{-n},\\
        \label{C}
		C(x)&\sim x^v\sum_{n=0}^\infty c_nx^{-n},
	\end{align}
	where $w$, $u$ and $v$ are real numbers such that $a_0\ne0$, $b_0\ne 0$ and $c_0\ne0$. 
	If the expansion of $B$ is infinite, then $B(A(x))$ is an asymptotic series if and only if $w\ge0$.

	For a given asymptotic series $B(x)$ and $C(x)$, we are trying to solve equation of the type
	\begin{equation}
		B(A(x))=C(x),
	\label{BA}
	\end{equation}
	where $A(x)$ is represented by its asymptotic series. The solution is posible under 
	reasonable conditions posed on series $B$ and $C$. An efficient recursive algorithm 
	for calculation of the coefficients $(a_n)$ will be derived.

	For the applications in integral means, the case $w=1$ is the only interesting case.
	But, we shall give an algorithm in more general situation.

	The two cases should be treated separately, the first one where $w> 0$, and the second one where $w<0$. Also, the value of the exponent $u$ is essential, the case $u=0$ and $u\ne0$ should be treated separately.

\subsection{The case $w>0$}

	The main result is given in the following theorem.

\begin{thm}
	\label{komp}
	Let $B(x)$ and $C(x)$ be asymptotic series given by \eqref{B} and \eqref{C}. 
	Suppose $u\ne0$.
	Then there exists asymptotic series $A(x)$ of the form \eqref{A} with $w>0$, such that it holds \eqref{BA} 
	if and only if the following conditions are satisfied:
	\begin{enumerate}
		\item $u$ and $v$ have the same sign and $w=v/u$ is a rational number,
		\item If $nw$ is not a positive integer then $b_n=0$,
		\item $b_0$ and $c_0$ are of the same sign.
	\end{enumerate}
	In that case, coefficients of the series $A(x)$ can be calculated using the following algorithm:
	\begin{align}
        \label{c0}
		a_0&= (c_0/b_0)^{1/u},\\
		a_n&= -\frac{a_0}{b_0 u P_0(u)} \bigg[
			\sum_{j=1}^{\lfloor n/w \rfloor}b_jP_{n-wj}(u-j)\\
			&+\frac{b_0}{n a_0}\sum_{k=1}^{n-1}[k(1+u)-n]a_kP_{n-k}(u)-c_n
			\bigg],
	\label{c1}
        \end{align}
	where $P_n$ are given by \eqref{pot}. 	
\end{thm}

\begin{proof}
	We shall give a constructive proof. We shall start with generalized asymptotic power series for $A(x)$ of the form \eqref{A}. 
	The coefficients of the series $A(x)$ will be recursively defined 
	such that \eqref{BA} is satisfied. This equation has the following form
	\begin{equation}
		\sum_{j=0}^\infty b_j (A(x))^{u-j}\sim
		x^v\sum_{n=0}^\infty c_nx^{-n}
	\end{equation}
	i.e.
	\begin{equation}
		\sum_{j=0}^\infty b_j x^{w(u-j)}\biggl(\sum_{k=0}^\infty a_kx^{-k}\biggr)^{u-j}
		\sim x^v\sum_{n=0}^\infty c_nx^{-n}
	\label{c3}
	\end{equation}
	Each power of the series for $A(x)$ can be transfered into asymptotic series using procedure given in \eqref{pot}. This can be made simultaneously, such that calculation of $P_{n}$ follows immediately after the coefficient $a_n$ is calculated. 
	In the algorithm which follows, $a_n$ will be calculated using previous coefficients $a_k$ for $k<n$ and $P_k$, also for $k<n$.
	This procedure gives
	\begin{gather}
		\sum_{j=0}^\infty b_jx^{w(u-j)}\sum_{k=0}^\infty P_k(u-j)x^{-k}
		\sim x^v\sum_{n=0}^\infty c_nx^{-n}\\
		\sum_{j=0}^\infty\sum_{k=0}^\infty b_jP_k(u-j)x^{wu-wj-k}
		\sim x^v\sum_{n=0}^\infty c_nx^{-n}.\label{komp-koef}
	\end{gather}
	From this equation we deduce the necessity of (1), since $x^{wu}$ and $x^{v}$ are terms with highest power. Therefore, $w=v/u$. If this number is not rational, then the powers from the left and right side cannot be equated. If $w$ is rational, then the product $wj$ must be an integer, for each index $j$ for which $b_j$ is not null. Hence, (2) must be fulfilled. 

	After cancelation of the power $x^v$ we obtain
	\begin{equation}
		\sum_{j=0}^{\lfloor n/w \rfloor}b_jP_{n-wj}(u-j)=c_n.
	\label{c4}
	\end{equation}
	For $n=0$ this reduces to
	$$
		b_0P_0(u)=c_0.
	$$
	Since from \eqref{pot} we have $P_0(u)=a_0^u$, the relation \eqref{c0} follows, and the condition (3) from Theorem.

	Coefficients $a_k$ can be calculated recursively from \eqref{c4}. 
	This will show that required conditions are also sufficient ones.
	Since $a_n$ is contained in a relation for $P_n$, we will extract this term from the sum in \eqref{c4}:
	$$
	\sum_{j=1}^{\lfloor n/w \rfloor}b_jP_{n-wj}(u-j)+
	b_0P_n(u)=c_n,
	$$
	and use \eqref{pot}, to obtain:
	$$
	\sum_{j=1}^{\lfloor n/w \rfloor}b_jP_{n-wj}(u-j)+
	\frac{b_0}{na_0}\sum_{k=1}^{n-1}[k(1+u)-n]a_kP_{n-k}(u)+
	\frac{b_0 u P_0(u)}{a_0}a_n=c_n.
	$$
	From here the main recursion \eqref{c1} immediatelly follows.
\end{proof}

	The algorithm in this form fails in the case $u=0$. But, if the series
	$B(x)$ has the form
	\begin{equation*}
		B(x)\sim b_0+\frac{b_1}x+\frac{b_2}{x^2}+\dots
	\end{equation*}
	then $B(A(x))=C(x)$ is possible only in the case $v=0$, when
	\begin{equation*}
		C(x)\sim c_0+\frac{c_1}x+\frac{c_2}{x^2}+\dots.
	\end{equation*}
	Now, we have that $b_0=c_0$ is necessary condition. Let $u'$ be the index of the first
	coefficient $b_k$, $k\ge1$ which is different from zero. Then $B(x)$ has the form
	\begin{equation*}
		B(x)=b_0+B_1(x),\qquad B_1(x)\sim x^{u'}\biggl[b_0'+\frac{b_1'}{x}+\frac{b_2'}{x^2}+\dots  \biggr]
	\end{equation*}
	where we denoted $b_0'=b_{u'}$, $b_1'=b_{u'+1}$ and so on. Similarly
	\begin{equation*}
		C(x)=b_0+C_1(x),\qquad C_1(x) \sim x^{w'}\biggl[c_0'+\frac{c_1'}{x}+\frac{c_2'}{x^2}+\dots  \biggr]
	\end{equation*}
	Now, the equation $B(A(x))=C(x)$ is equivalent to $B_1(A(x))=C_1$. But leading exponent in this equation is $u'\ne0$ and the problem is reduced to the case covered by Theorem~\ref{komp}. Therefore, for the computational point of view, the following 
	case:
	\begin{equation*}
		B(x)\sim d+x^u\sum_{n=0}^\infty b_nx^{-n},\qquad
		C(x)\sim d+x^v\sum_{n=0}^\infty c_nx^{-n},
	\end{equation*}
	is for integer $u$, $u<0$ equivalent to the one covered in Theorem~\ref{komp} and the same algorithm can be applied. In fact, the same conclusion holds for every noninteger $u$.

\subsection{Case $w<0$}

	Let us cover this case too. It is not important in our applications, but it can be interesting from another point of view.

\begin{thm}
		Let $B(x)$ and $C(x)$ be asymptotic series given by \eqref{B} and \eqref{C}. Then there exists series $A(x)$ of the form \eqref{A} where $w<0$ such that \eqref{BA} holds if and only if the following conditions hold:
		\begin{enumerate}
			\item
				representation of $B$ is finite, ie.\ 
				\begin{equation}
					B(x)\sim x^u\sum_{n=0}^M b_nx^{-n}
				\end{equation} 
				where $M$ is such that $b_M \neq 0$.
			\item
				\begin{equation}
					w=\frac{v}{u-M}<0
				\end{equation}
		\end{enumerate}
		With that properties satisfied, coefficients of series $A(x)$ can be calculated using recursive formula:
		\begin{align}
		 a_0 &= \left( \frac{c_0}{b_M} \right)^{M-u},\\
		 a_n &=  \frac{a_0^{1-u+M}}{(u-M)b_M} \bigg( c_n-\frac{b_M}{n a_0}\sum_{k=1}^{n-1}\left[k(1+u-M)-n\right]a_kP_{n-k}(u-M)\\
		 &\qquad\qquad-\sum_{j=1}^{\min\{M,-n/w\}}b_{M-j}P_{n+wj}(u-M+j) \bigg)
		\end{align} 
\end{thm}

\begin{proof}
	The proof is similar to the proof of Theorem \ref{komp}. From the relation \eqref{komp-koef} we see that term with the highest power of $x$ on the left side has the form $b_ja_0^{u-j}x^{wu-wj}$. Since $w$ is negative and $j\geq0$, there must exist some $M$ such that $b_n=0$ for $n>M$ and $b_M\neq0$ to obtain the maximum. It follows that $w=\frac{v}{u-M}$. After arranging sums in the \eqref{komp-koef} we get
	\begin{equation*}
		\sum_{n=0}^{\infty}\sum_{j=0}^{\min\{M,-n/w\}}b_{M-j}P_{n+wj}(u-M+j)x^{-n}
		\sim\sum_{n=0}^\infty c_n x^{-n}
	\end{equation*}
	and then
	\begin{equation*}
		\sum_{j=1}^{\min\{M,-n/w\}}b_{M-j}P_{n+wj}(u-M+j)+b_M P_n(u-M)= c_n.
	\end{equation*}
	Extracting $a_n$ from $P_n(u-M)$ follows the statement.
\end{proof}

\begin{exa}
	Let us find asymptotic series which satisfies an equation
	$A^2(x)+1/A(x)=x^2$. Here $B(x)=x^2+1/x$, $C(x)=x^2$. Hence, 
	for the first solution we have $u=2$, $v=2$, $w=1$ and (2.9) will give
	\begin{equation*}
		A(x)\sim x\bigg(1-\frac1{2x^3}-\frac{3}{8x^6}-\frac1{2x^9}
		-\frac{105}{128x^{12}}
		-\frac{3}{2x^{15}}
		-\frac{3003}{1024x^{18}}-\frac{6}{x^{21}}+\dots\bigg)
	\end{equation*}

	The expansion of the function $B$ is finite, and this enables another solution for which $M=3$, hence $w=-2$ and (2.17) gives
	\begin{equation*}
		A(x)\sim \frac1{x^2}\bigg(1+\frac1{x^6}+\frac3{x^{12}}+\frac{12}{x^{18}}
		+\frac{55}{x^{24}}
		+\frac{273}{x^{30}}
		+\frac{1428}{x^{36}}-\frac{7752}{x^{42}}+\dots\bigg).
	\end{equation*}
\end{exa}	

\section{Applications to integral means}

        The Theorem~\ref{komp} can be applied to
        \begin{equation}
	\begin{aligned}
        A(x)&=I_f(x+s,x+t)\sim x\sum_{n=0}^\infty a_nx^{-n},\\
	B(x)&=f(x)\sim x^u\sum_{n=0}^\infty b_nx^{-n},\\
        C(x)&=\frac{1}{t-s}\int_s^t f(x+u)\du \sim x^v\sum_{n=0}^\infty c_nx^{-n}.
	\end{aligned}
        \end{equation}
        In order to use this theorem, it is necessary to know asymptotic 
	expansion of the function $f$ and of the integral $C(x)$.
        The coefficients in the expansion of the function $C(x)$ can be easily calculated, but the form of this coefficients depends on the fact whether $B(x)$ contains the term $1/x$ or not.
	Notice that integral means of the functions $f$ and $\lambda f$ coincide for every constant $\lambda\ne0$. Therefore, we may assume that
	it holds $b_0=1$. From this, it will also be satisfied that $c_0=1$ and $a_0=1$.

	Let us determine coefficients $c_n$ when $u$ is not an integer or $u\leq-2$.
\begin{align*}
	C(x)&=B(A(x))=f(I_f)=\frac{1}{t-s}\int_{x+s}^{x+t} f(z)\dif z \\
	&\sim\frac{1}{t-s}\int_{x+s}^{x+t} \sum_{n=0}^{\infty}b_n z^{u-n} \dif z \\
	&\sim\frac{1}{t-s}\sum_{n=0}^{\infty} \frac{b_n}{u-n+1} ((x+t)^{u-n+1}-(x+s)^{u-n+1}) \\
	&\sim\frac{1}{t-s}\sum_{n=0}^{\infty} \frac{b_n}{u-n+1} 
	\sum_{k=1}^{\infty}\binom{u-n+1}{k} (t^k-s^k)x^{u-n+1-k}\\
	&\sim\sum_{m=1}^{\infty}\sum_{n=0}^{m-1} \frac{b_n}{u-n+1}  
	\binom{u-n+1}{m-n} \frac{t^{m-n}-s^{m-n}}{t-s} x^{u+1-m}\\
	&\sim x^u \sum_{n=1}^{\infty}\sum_{k=0}^{n-1} \frac{b_k}{u-k+1}  
	\binom{u-k+1}{n-k} \frac{t^{n-k}-s^{n-k}}{t-s} x^{-n+1}\\
	&\sim x^u \sum_{n=0}^{\infty} 
			\left( \sum_{k=0}^{n} \frac{b_k}{u+1-k}  
		\binom{u+1-k}{n+1-k} \frac{t^{n+1-k}-s^{n+1-k}}{t-s}  \right)
		 x^{-n}\\
	&\sim x^u \sum_{n=0}^{\infty} 
			\left( \sum_{k=0}^{n} \frac{b_k}{n+1-k}  
		\binom{u-k}{n-k} \frac{t^{n+1-k}-s^{n+1-k}}{t-s}  \right)
		 x^{-n}.
\end{align*}

	Note that we obtained $v=u$ and therefore the leading term in asymptotic expansion of mean is $x$ (i.e. $w=1$).
	
	We will now cover the general case when term with power $x^{-1}$ is a member of asymptotic expansion of the function $f$. It is the term with coefficient $b_{u+1}$, whenever $f$ has a standard representation
	$$
		f(x)\sim x^u\sum_{n=0}^\infty b_n x^{-n}
	$$
	where $u$ is an integer, $u\ge -1$.

	The difference between this and previous case is different asymptotic expansion of the function $C(x)$.
	The value of integral of the function $1/x$ will be denoted as
	$$
		D(x)=\frac1{t-s}\int_s^t\frac{\dif z}{x+z}=\frac1x\cdot
		\sum_{n=0}^\infty d_n x^{-n}.
	$$
	Then it is easy to derive
	\begin{equation}
	\label{coef-d}
		d_n=\frac{(-1)^{n}}{n+1}\cdot\frac{t^{n+1}-s^{n+1}}{t-s}.
	\end{equation}
	
	In this case, the coefficients $\overline c_n$ of the function $C(x)$ have the following form
		\begin{equation}
		\overline c_n= \left\{
		\begin{array}{cl}
			c_n, & 0\leq n\leq u,\\
			b_{u+1}, & n=u+1,\\
			c_n'+b_{u+1}d_{n-u-1}, & n\geq u+2,
		\end{array}
		 \right. 
	\end{equation}
	where $(c_n)$ is defined by \eqref{coef-cn1}, $(d_n)$ is defined by \eqref{coef-d} and
	\begin{equation}
		c_n'=\sum_{k=u+2}^{n} \frac{b_k}{n+1-k} \binom{u-k}{n-k}\frac{t^{n+1-k}-s^{n+1-k}}{t-s}.
	\label{coef-cn2}
	\end{equation}
	To explain this, it is sufficient to note that term $b_{u+1}x^{-1}$ has no influence to the coeficient with nonnegative powers, hence
	$\overline c_n=c_n$ for $n\le u$. The only term with power $x^{0}$ can arise from the logarithm function obtained by integration of this term.
	The coeficient of this term is $\overline c_{u+1}$ and it is equal to $b_{u+1}$.
	If $n\ge u+2$, then $
	\overline c_n$ is the sum two parts, one obtained from integration of negative power of $x$, and the second one from the expansion of the logarithm function.

	But, by the direct inspection of the coefficients, it is evident that $\overline c_n$ is equal to $c_n$ for al values of $n$.

	Therefore, we can summarize this discussion in the form of a theorem.

\begin{thm}
        \label{tm1}
	Let function $f$ has the folowing asymptotic expansion 
        \begin{equation}
                f(x)\sim x^u\sum_{n=0}^{\infty}b_nx^{-n}.
        \end{equation}
	Then integral mean has the form
	\begin{equation}
                I_f(x+s,x+t)\sim x \sum_{n=0}^{\infty}a_n x^{-n}
        \label{intmean}
	\end{equation} 
        and coefficients $a_n$ satisfy the following recursive relation:
	\begin{align}
		a_0&=1,\\
		a_n&=-\frac{1}{b_0 u} \bigg(
		\sum_{j=1}^{n}b_j P_{n-j}(u-j)\nonumber\\
	&\quad \quad+
		\frac{b_0}{n} \sum_{k=1}^{n-1}(k(1+u)-n)a_kP_{n-k}(u)-c_n 
		\bigg),
	\end{align}
	where $(P_n)$ are defined by \eqref{pot}
	and
	\begin{equation}
	c_n= \sum_{k=0}^{n} \frac{b_k}{n+1-k}  
	\binom{u-k}{n-k} \frac{t^{n+1-k}-s^{n+1-k}}{t-s}.
	\label{coef-cn1}
	\end{equation}
\end{thm}

\section{Examples}

\begin{exa}
	Asymptotic expansion of generalized logarithmic mean.
	Suppose $f(x)=x^r$, $r\ne0$, $r\ne-1$. Then integral mean has the form
	\eqref{intmean} where $a_0=1$ and
	\begin{equation}
		a_n=-\frac{1}{r}\left(\frac1n\sum_{k=1}^{n-1}(k(1+r)-n)a_kP_{n-k}(r)-c_n\right)
	\end{equation}
	where $(P_n)$ are defined by \eqref{pot}
	and
	\begin{equation}
	c_n= \frac{1}{n+1}  
	\binom{r}{n} \frac{t^{n+1}-s^{n+1}}{t-s}.
	\end{equation}
\end{exa}

	See \cite{EV} for another approach to this problem.
	The first few coefficients are
	\begin{equation}
	\label{genmean2}
	\begin{aligned}
	a_0&=1,\\
	a_1&=\a,\\
	a_2&=\tfrac{1}{6}(r-1)\b^2,\\
	a_3&=-\tfrac{1}{6}(r-1)\a\b^2,\\
	a_4&=\tfrac{1}{360}(r-1)\b^2[(-2r^2-5r+13)\b^2+60\a^2].\\
	\end{aligned}
	\end{equation}

\begin{exa}
	Polygamma function $\psi^{(m)}$, $m\geq 1$ has following asymptotic expansion
	\begin{equation}
		\psi^{(m)}(x) \sim (-1)^{m-1}\left(\frac{(m-1)!}{x^m}+\frac{m!}{2x^{m+1}}+\sum_{n=1}^\infty \frac{(2n+m-1)!}{(2n)!}B_{2n} x^{-2n-m}\right).
	\end{equation}
	Therefore,
	\begin{align*}
	I_{\psi^{(m)}}(x+s,x+t)\sim x+\a-\frac16(m+1)\b^2x^{-1}+\frac{1}{12}(m+1)(2\a-1)\b^2 x^{-2}+&\\
	+\frac{1}{360}(m+1)(2m^2\b^2-5m\b^2-13\b^2-60\a^2+60\a+5m-5)\b^2x^{-3}+\cdots&
	\end{align*}

	This and similar results are obtained in \cite{CEV2}, where some properties of polynomial coefficients are also discussed.
\end{exa}

\begin{exa} 
	Let us choose $f(x)=\dfrac1{x^2}+\dfrac1{x^3}$. Because of the difficulty with finding inverse formula, the exact value of integral mean will not be calculated.
	Using algorithm described in the Theorem~\ref{tm1},
	we can write
	$$
		I_f(x+s,x+t)
		\sim
		x\biggl(\sum_{n=0}^\infty a_nx^{-n}\biggr)
	$$
	where
	\begin{align*}
		a_0&=1,\\
		a_1&=\a,\\
		a_2&=-\tfrac12\b^2,\\
		a_3&=\tfrac{1}{4}\b^2(-1+2\a),\\
		a_4&=-\tfrac{1}{8}\b^2(-3-4\a+4\a^2+\b^2),\\
		a_5&=\tfrac{1}{16}\b^2(8\a^3-12\a^2+6\a(-3+\b^2)-3(3+\b^2)),\\
		\vdots  
	\end{align*}

	Since 
	$$
		\frac1{t-s}\int_s^tf(x+u)du=f(I_f(x+s,x+t))
	$$
	we can check the numerical accuracy of these formulae. Let us take a decent (relatively small) values for integral bounds, for example:
	$$
		I=\frac1{10}\int_{100}^{110}\biggl(\dfrac1{x^2}+\dfrac1{x^3}\biggr)dx.
	$$
	Then, the best choice is $x=105$, $s=-5$, $t=5$ since this will give $\a=0$ and $\b=5$. We have
	$$
		I_f(x-t,x+t)\sim I_4=x-\frac{\b^2}{2x}-\frac{\b^2}{4x^2}
		-\frac{\b^2(\b^2-3)}{8x^3}-\frac{3\b^2(\b^2+3)}{16x^4}
	$$
	Now we can compare:
	\begin{align*}
		I&=9.1776859504\cdot 10^{-5},\\
		f(I_4)&=9.1776859416\cdot 10^{-5}
	\end{align*}
	so the relative error is of order $9.53\cdot10^{-10}$.

	The same precision can be obtained using Simpson rule with 16 nodes. Adding additional terms
	\begin{align*}
		I_5&=I_4+\frac{\b^2(27+ 18\b^2-2\b^4)}{32x^5},\\
		I_6&=I_5-\frac{\b^2(81+ 78\b^2+ 10\b^4)}{64x^6}
	\end{align*}
	the relative errors correspond to Simpson rule with 32 and 104 nodes.

	Of course, taking greater value for $x$, the precision of asymptotic methods increases rapidly. For example, $x=1005$ with other values unchanged gives relative error for $f(I_6)$ of order $5.5\cdot10^{-21}$, while Simpson rule with 32 nodes has error of order $6.2\cdot 10^{-15}$.
\end{exa}

\section{Logarithmic case}

	The theory developed so far does not cover all important applications. For example, the digamma case described in introduction is not covered yet. The reason for this is the logarithm which is the part of the function $f$. Let us cover this case too.

	Suppose that function $f$ has the following asymptotic expansion
	\begin{equation}
		f(x)\sim b\cdot\log x+x^{-1}\sum_{n=0}^\infty b_nx^{-n}.	
	\end{equation}
	Here, for the case of simplicity, the choice $u=-1$ was made. However, any integer $u<0$ will lead to succesfull algorithm.

	The logarithm has influence in the expansion of the integral function $C(x)$ which is easy to describe,
	\begin{align*}
		g(x)&=\frac{1}{t-s}\int_s^t \log (x+z)\dif z\\
		&=\frac{(x+t)\log(x+t)-(x+s)\log(x+s)}{t-s}-1\\
                &=\log x+\frac1{t-s}\biggl[
                (x+t)\log\Bigl(1+\frac tx\Bigr)-(x+s)\log\Bigl(1+\frac sx\Bigr)\biggr]-1\\
                &=\log x+\sum_{n=1}^\infty(-1)^{n-1}\frac{t^{n+1}-s^{n+1}}{n(n+1)(t-s)}\,x^{-n}\\
                &=\log x-\sum_{n=1}^\infty\frac{d_n}n\,x^{-n}.
        \end{align*} 
	where $(d_n)$ are defined by \eqref{coef-d}.
	This expansion multiplied by $b$ should be added to the usual expansion given in Theorem~\ref{tm1}.

	On the other hand, $f(A(x))$ has an additional term:
	$$
		b\log(A(x))=b\log x+b\log\bigg(\sum_{n=0}^\infty a_nx^{-n}\bigg).
	$$
	The term $b\log x$ will be canceled from the both sides of equations, and the rest can be arranged as before in a series of recursive relations. Here, we use the transformation given in Lemma~\ref{lemma-log}:
	\begin{align*}
		f(A(x))&\sim b \log(A(x))+\sum_{n=0}^\infty b_n(A(x))^{-n-1}\\
		&\sim b\log x+ b\log\biggl(\sum_{n=0}^\infty a_nx^{-n}\biggr)
		+\sum_{n=1}^\infty b_{n-1}\biggl(\sum_{k=0}^\infty a_kx^{-k}\biggr)^{-n}x^{-n}\\
		&\sim b\log x+b\sum_{n=1}^\infty L_nx^{-n}+\sum_{n=1}^\infty
		b_{n-1}\biggl(\sum_{k=0}^\infty P_k(-n)x^{-k}\biggr)x^{-n}\\
		&\sim b\log x+b\sum_{n=1}^\infty L_nx^{-n}+\sum_{n=1}^\infty
		\biggl(\sum_{k=0}^{n-1}b_k P_{n-1-k}(-k-1)\biggr)x^{-n}\\
		&\sim b\log x-b\sum_{n=1}^\infty\frac{d_n}n\,x^{-n}+\sum_{n=1}^\infty c_{n-1}x^{-n}.
	\end{align*}
	We should now equate coefficients of the term $x^{-n}$. The highest index of the sequence $(a_n)$ which appear here is in the $L_n$ term. Since $a_0=1$, we have
	\begin{equation*}
		L_n=a_n-\frac1n\sum_{k=1}^{n-1}kL_ka_{n-k}.
	\end{equation*}
	Hence, we finally obtain
	\begin{equation*}
		b\biggl(a_n-\frac1n\sum_{k=1}^{n-1}L_ka_{n-k}\biggr)
		+\sum_{k=0}^{n-1}b_k P_{n-1-k}(-k-1)=-\frac{bd_n}{n}+c_{n-1}.
	\end{equation*}
	From here, $a_n$ can be recursively calculated.
	The exact algorithm is given in the next theorem.

\begin{thm}\label{log2}
	Let
	\begin{equation}
		f(x)\sim b\log x+x^{-1}\sum_{n=0}^\infty b_nx^{-n}.
	\end{equation}
	Then coefficients from asymptotic expansion
	\begin{equation}
		I_f(x+s,x+t)\sim x\sum_{n=0}^\infty a_nx^{-n}
	\end{equation}
	can be calculated as follows
	\begin{align}
		a_0&=1,\\
		a_n&=-\frac{d_n}{n}+\frac{c_{n-1}}{b}	+\frac1n \sum_{k=1}^{n-1}k L_k a_{n-k}-\frac1b\sum_{k=0}^{n-1}b_k P_{n-1-k}(-k-1),
	\end{align}
	where
	\begin{equation*}
		L_n=a_n-\frac1n \sum_{k=1}^{n-1}k L_k a_{n-k},
	\end{equation*}
	and $(P_n)$ is given in Lemma~\ref{pot}.
\end{thm}

\begin{exa}
Asymptotic expansion of $\psi$ function is known
	$$
	\psi(x)\sim \log x+x^{-1}\sum_{n=0}^{\infty} \frac{(-1)^nB_{n+1}}{n+1}x^{-n}.
	$$
	Asymptotic expansion of integral mean of digamma function can easily be calculated using Theorem \ref{log2}
	\begin{align*}
		I_{\psi}&\sim  x+\a-\frac{\b^2}{6}x^{-1}+\frac{(2\a-1)\b^2}{12}x^{-2}
		 -\frac{\left(5+60\a(\a-1)+13\b^2\right)\b^2}{360}x^{-3}+\cdots
	\end{align*}
	Of course, we obtain results from \eqref{coef-psi}.
\end{exa}

\begin{exa}
	Another extreme is the function $f(x)=\log x$, when all coefficients $b_k$ are equal to zero. The corresponding integral mean is identric mean $I(s,t)$.
	Applying the algorithm of Theorem~\ref{log2}, we obtain the following coefficients:
	\begin{equation*}
		I(s,t)=I_{\log x}(s,t)\sim x\sum_{n=0}^\infty a_kx^{-k}
	\end{equation*}
	where
	\begin{align*}
		a_0&=1,\\
		a_1&=\a,\\
		a_2&=-\tfrac16\b^2,\\
		a_3&=\tfrac{1}{6}\a\b^2,\\
		a_4&=-\tfrac{1}{360}\b^2(60\a^2+13\b^2),\\
		a_5&=\tfrac{1}{120}\a\b^2(20\a^2+13\b^2),\\
		a_6&=-\tfrac1{45360}\b^2(7560\a^4+9828\a^2\b^2+737\b^4).
	\end{align*}
	
	Identric mean is special case of generalized logarithmic mean, obtained taking a limit $r\to0$. So, it is not a surprise that these coefficients coincide to one calculated in \eqref{genmean2} if we choose there $r=0$. 
\end{exa}

\section{Integral mean of Wallis quotient and Wallis power function}

	Wallis quotient is the name for the ratio of two gamma functions:
	\begin{equation*}
		W(x,t,s)=\frac{\Gamma(x+t)}{\Gamma(x+s)}
	\end{equation*}
	and Wallis power function is defined as
	\begin{equation*}
		F(x,t,s)=\biggl(\frac{\Gamma(x+t)}{\Gamma(x+s)}\biggr)^{1/(t-s)}.
	\end{equation*}
	These two functions have important role in various parts of applied and theoretical mathematics.

	An efficient algorithm for asymptotic expansion of these functions were derived recently by the first author and T.\ Buri\'c, see \cite{BE1,BE2} for details. We shall use now derived algorithm for calculation of integral mean of these functions.

	Let us start with Wallis power function which has more natural expansion. We should change notation of arguments in order not to interfere with bounds of integral mean.
	In \cite{BE1} was proved that Wallis power function
	\begin{equation}
		F(x,t',s')=\left[ \frac{\Gamma(x+t')}{\Gamma(x+s')}\right]^{\frac1{t'-s'}}
	\end{equation}
	has the following asymptotic expansion
	\begin{equation}
		F(x,t',s')\sim x+\sum_{n=0}^{\infty}Q_{n+1}(\alpha',\beta')x^{-n}.
	\end{equation}
	where $\a'$ and $\b'$ are internal variables defined by
	$$
		\a'=\frac12(s'+t'+1),\qquad \b'=\frac14[1-(t'-s')^2],
	$$
	and $(Q_n)$ are polynomials:
        \begin{align}
	\label{pol-Q}
	 	Q_0&=1\nonumber\\
        Q_1&=\a'\nonumber\\
        Q_2&=\tfrac16\b'\nonumber\\
        Q_3&=-\tfrac16\a'\b'\\
        Q_4&=\tfrac16\a'^2\b'-\tfrac1{60}\b'-\tfrac{13}{360}\b'^2\nonumber\\
        Q_5&=-\tfrac16\a'^3\b'+\tfrac1{20}\a'\b'+\tfrac{13}{120}\a'\b'^2\nonumber\\
        Q_6&=\tfrac16\a'^4\b'-\tfrac1{10}\a'^2\b'-\tfrac{13}{60}\a'^2\b'^2
        +\tfrac1{126}\b'+\tfrac{53}{2520}\b'^2+
        \tfrac{737}{45360}\b'^3\nonumber
        \end{align}

	Now we apply described procedure from Theorem~\ref{komp} on $b_n=Q_n(\alpha',\beta')$ and $u=1$ to obtain
	\begin{align*}
		c_0&=1,\\
		c_1&=\a+\a',\\
		c_2&=\frac{\b'}{6},\\
		c_3&=-\frac{1}{6}(\a+\a')\b',\\
		c_4&=\frac{1}{360}\left(-6+60(\a+\a')^2+20\b^2-13\b'\right)\b',\\
		\vdots
	\end{align*}
	Using recursive formula from Theorem \ref{komp} we calculate coefficients in asymptotic expansion of the integral mean $I_F$. Here are the first few terms:
	\begin{align*}
		 a_0&=1,\\
		 a_1&=\a,\\
		 a_2&=0,\\
		 a_3&=0,\\
		 a_4&=\frac{1}{18}\b^2\b',\\
		 a_5&=-\frac16(\a+\a')\b^2\b'\\
		 a_6&=\frac{1}{270}\left( -9+90(\a+\a')^2+9\b^2-17\b'\right)\b^2\b'\\
		 a_7&=-\frac{1}{54}(\a+\a')(-9+30(\a+\a')^2+9\b^2-17\b')\b^2\b'\\
		 &\vdots
	\end{align*}

	Although we have here four parameters, the first two essential coefficients $a_2$ and $a_3$ disappear. We will explain this situation later.
	
	Let us now derive similar expansion for the Wallis ratio.
	It has the following asymptotic expansion (see \cite{BE2})
	\begin{equation}
		W(x,t',s')=\frac{\Gamma(x+t')}{\Gamma(x+s')}\sim \sum_{n=0}^\infty C_n(\a',\b')x^{\c'-n}
	\end{equation} 
	where $\c'=t'-s'$ and the first few polynomials $C_n(\a',\b')$ are 
	\begin{align*}
	\label{pol-C}
		C_0&=1,\\
        C_1&=\a'\c',\\
        C_2&=\frac16\c'(\b'+3\a'^2(\c'-1)),\\
        C_3&=\frac16\a'\c'(\b'+3\a'^2(\c'-1))(\c'-2).
        \end{align*}
	see \cite{BE1}. 
	We apply Theorem \ref{tm1} and obtain
	\begin{equation*}
		I_W(x+s,x+t)\sim x\sum_{n=0}^\infty a_nx^{-n}
	\end{equation*}
	where
	\begin{align*}
		a_0&=1,\\
		a_1&=\a,\\
		a_2&=\frac16\b^2(\c'-1),\\
		a_3&=-\frac16\b^2(\a+\a')(\c'-1),\\
		a_4&=-\frac{1}{360} \b^2 \big(60(\a+\a')^2+13\b^2-40\b'\\
		&\qquad-2 \big(30(\a+\a')^2+9\b^2-10\b'\big)\c'+3\b^2 \c'^2+2\b^2\c'^3\big),\\
		a_5&=-\frac{1}{120}\b^2 (\a+\a') \big(-20 (\a+\a')^2-13\b^2+40 \b'\\
		&\qquad+2\big(10 (\a+\a')^2+9\b^2-10 \b'\big)\c'-3\b^2\c'^2-2\b^2 \c'^3\big).
	\end{align*}

\section{Analysis of the asymptotic expansions}

	Using the algorithm given in Theorem~\ref{tm1}, the value of coefficients $(c_n)$ can be described in terms of $\a$, $\b$ and coefficients $b_n$, $n\ge1$. Here is the result. We shall restrict ourselves to the functions covered by Theorem~\ref{tm1}.
	For a function of the form
	$$
		f(x)\sim 
		x^u\biggl(1+\frac {b_1}x+\frac{b_2}{x^2}+\dots\biggr)
	$$
	we obtain
	\begin{align*}
		a_0&=1,\\
		a_1&=\a,\\
		a_2&=\frac16(u-1)\b^2,\\
		a_3&=-\frac1{6u}(u-1)\b^2(u\a+b_1)\\
		a_4&=\frac1{360}\b^4(-2u^3-3u^2+18u-13)\\
		&\qquad\quad+\frac{b_1^2\b^2}{6u^2}(u-1)^2-\frac{\b^2b_2}{3u}(u-2).
	\end{align*} 
	$a_5$ is too ugly to be writen down. Hence, the property \eqref{c3} is the consequence of the fact that all classical integral means have coefficient $b_1$ equal to zero. 

	Interesting to say, the general power mean $M_r(s,t)=[\frac12(s^r+t^r)]^{1/r}$,
	and the harmonic mean $H(s,t)=2st/(s+t)$ are not integral means, but they also have the property \eqref{c3}. See \cite{EV} for their expansions.

	We see now that integral mean of an integral mean always has coefficients $a_2$ and $a_3$ equal to zero, since corresponding value of $u$ is $u=1$. This will be the case for all functions whose asymptotic expansion begins with $x$, the Wallis power function is good example.

	The numerical calculation shows that the third iteration of integral means has four zero coefficients $c_2=c_3=c_4=c_5=0$, and so on. We do not see a good reason for trying to find the proof of this fact.

\end{document}